\documentclass[12pt,draft]{amsart}

\usepackage{amsmath}
\usepackage{amssymb}
\usepackage{array}
\usepackage{amsthm}

\renewcommand{\atop}[2]{%
\genfrac{}{}{0pt}{}{#1}{#2}}

\newcommand{\cf}[2]{%
\genfrac{}{}{0pt}{}{#1}{#2}}

\newtheorem{theorem}{Theorem}%[section]
\newtheorem{lemma}{Lemma}
\newtheorem{corollary}{Corollary}%[section]
\newtheorem{proposition}{Proposition}

\theoremstyle{definition}

\renewcommand{\atop}[2]{%
\genfrac{}{}{0pt}{}{#1}{#2}}

\renewcommand{\cf}[2]{%
\genfrac{}{}{0pt}{}{#1}{#2}}

\begin{document}

\title{On a continued fraction expansion for Euler's constant.}

% author one information
\author{Kh.~Hessami Pilehrood}

% author two information
\author{T.~Hessami Pilehrood}

% both authors have the same address and the same current address

\address{\begin{flushleft} School of Mathematics, Institute for Research  in Fundamental Sciences
(IPM), P.O.Box 19395-5746, Tehran, Iran \end{flushleft}}

\email{hessamik@gmail.com,  hessamit@gmail.com}
\noindent\curraddr{
Department of Mathematics and Statistics, Dalhousie University, Halifax, Nova Scotia,
B3H 3J5, Canada
}

%\subjclass{Primary 11J70, 11B37, 11B65, 11Y65; Secondary 33C60, 33C45, 33F10.}

\date{}


\begin{abstract}
Recently, A.~I.~Aptekarev and his collaborators found a sequence of rational approximations to
Euler's constant $\gamma$ defined by a third-order homogeneous linear recurrence.
In this paper, we give a new interpretation of Aptekarev's approximations
in terms of Meijer $G$-functions and hypergeometric-type series.
This approach allows us to describe a very general construction giving linear forms in $1$ and $\gamma$
with rational coefficients. Using this construction we find new rational approximations to $\gamma$
generated by a second-order inhomogeneous linear recurrence with polynomial coefficients. This leads
to a continued fraction (though not a simple continued fraction) for Euler's constant.
It seems to be the first non-trivial continued fraction expansion convergent to Euler's constant sub-exponentially,
 the elements  of which can be expressed as a general pattern. It is interesting to note that the same
 homogeneous recurrence generates a continued fraction for the Euler-Gompertz constant found by Stieltjes
 in 1895.
\end{abstract}

\maketitle

\section{Introduction}

In 1978,
 R.~Ap\'ery \cite{ap,po} stunned the mathematical world with a
proof that
$\zeta(3)=\sum_{n=1}^{\infty}\frac{1}{n^3}$ is irrational. Since then many
different proofs of this fact have appeared in the literature (see \cite{fi} and the references given there).
The main idea of all the known proofs is essentially the same and consists in constructing
a sequence of linear forms $I_n=u_n\zeta(3)-v_n,$ $n=0,1,2,\ldots,$ which satisfy
the following conditions:
$$
\limsup_{n\to\infty} |I_n|^{1/n}\le (\sqrt{2}-1)^4=0,0294372\ldots,
$$
$I_n\ne 0$ for infinitely many $n,$ and $u_n\in{\mathbb Z},$ $2D_n^3v_n\in {\mathbb Z},$
where $D_n$ is the least common multiple of the numbers $1,2,\ldots,n.$
If we suppose that $\zeta(3)$ were a rational number $a/b,$ then $2bD_n^3I_n$ is a non-zero integer
for infinitely many $n$ and, on the other hand, it tends to zero as $n\to\infty$ (since
$D_n^{1/n}\to e$ and $e^3(\sqrt{2}-1)^4=0.591\ldots<1$), which is a contradiction.

The diversity of all the proposed proofs of the irrationality of $\zeta(3)$
is presented by different interpretations and ways of obtaining the linear forms $I_n$
and its various representations.
Ap\'ery showed that the sequences $u_n$ and $v_n$ are given by the formulae
\begin{equation*}
u_n=\sum_{k=0}^n\binom{n}{k}^2\binom{n+k}{k}^2, \quad
v_n=\sum_{k=0}^n\binom{n}{k}^2\binom{n+k}{k}^2\left(\sum_{m=1}^n\frac{1}{m^3}+
\sum_{m=1}^k\frac{(-1)^{m-1}}{2m^3\binom{n}{m}\binom{n+m}{m}}\right)
\end{equation*}
and satisfy the second-order linear recurrence relation
$$
(n+1)^3y_{n+1}-(34n^3+51n^2+27n+5)y_n+n^3y_{n-1}=0
$$
with the initial conditions $u_0=1, u_1=5,$ $v_0=0,$ $v_1=6.$
This implies immediately that $v_n/u_n$ is the $n$th convergent
of the following continued fraction:
$$
\zeta(3)=
\frac{6}{5}\cf{}{-}
\frac{1}{117}\cf{}{-}\frac{64}{535}\cf{}{-}%\frac{5}{z^2}\cf{}{+}
\ldots\cf{}{-}\frac{n^6}{34n^3+51n^2+27n+5}\cf{}{-}\ldots.
$$
The shorter proof of Ap\'ery's theorem  has been found in 1979 by F.~Beukers \cite{be},
who used multiple Euler-type integrals and Legendre polynomials.

In 1996 inspired by the works of F.~Beukers \cite{be1} and L.~Gutnik \cite{gu},
Yu.~V.~Nesterenko \cite{ne}
proposed another proof of the irrationality of $\zeta(3)$ and a new expansion
of this number into continued fraction. His proof was based on
the hypergeometric-type series
\begin{equation}
-\frac{1}{2}\sum_{k=1}^{\infty}R_n'(k)=-\frac{1}{2}\sum_{k=1}^{\infty}\frac{d}{dt}
\left.\frac{\Gamma^4(t)}{\Gamma^2(t-n)\Gamma^2(t+n+1)}\right|_{t=k}=u_n\zeta(3)-v_n
\label{eq01}
\end{equation}
that can be written (by the residue theorem) as a complex integral or a Meijer $G$-function
(see \cite[Section 5.2]{lu}, for definition)
\begin{equation}
\begin{split}
-\sum_{k=1}^{\infty}R_n'(k)&=\frac{1}{2\pi i}\int\limits_{c-i\infty}^{c+i\infty}
R_n(s)\left(\frac{\pi}{\sin\pi s}\right)^2\,ds=\frac{1}{2\pi i}\int\limits_{c-i\infty}^{c+i\infty}
\frac{\Gamma^2(n+1-s)\Gamma^4(s)}{\Gamma^2(n+1+s)}\,ds\\
&=G_{4,4}^{4,2}\left(
\left.\atop{-n, -n,
n+1, n+1}{0, \, 0,\, 0,\,0}\right|1
\right),
\end{split}
\label{eq02}
\end{equation}
here $c$ is an arbitrary real number satisfying $0<c<n+1$ and $R_n(t)$ is a rational function
defined by
$$
R_n(t)=\frac{(t-1)^2\cdots(t-n)^2}{t^2(t+1)^2\cdots(t+n)^2}.
$$
 As we can see now,
 the hypergeometric construction
 (\ref{eq01}), (\ref{eq02}) appeared to be more transparent for generalizations
 on obtaining irrationality results for other odd zeta values (see \cite{ri}, \cite{zu}).

 Euler's constant was first introduced by Leonhard Euler in 1734 as
 $$
 \gamma=\lim_{n\to\infty}\left(1+\frac{1}{2}+\frac{1}{3}+\cdots+
 \frac{1}{n}-\log n\right)=0.57721566490153286\ldots.
 $$
 It can be considered as an analogue of the value ''$\zeta(1)$`` of Riemann's zeta
 function if we compensate the partial sums of the divergent harmonic series
 by the natural logarithm. It is not known if $\gamma$ is an irrational or transcendental
 number. The question of its irrationality
 remains a famous unresolved problem in the theory of numbers. Even obtaining good rational approximations
 to it was unknown until recently. First such approximations defined by a third-order linear
 recurrence were found by A.~I.~Aptekarev and his collaborators \cite{apt} in 2007. More precisely,
 the numerators $\tilde{p}_n$ and denominators~$\tilde{q}_n$ of these approximations are positive integers
 generated by the  recurrence relation
 \begin{equation}
 \begin{split}
(16n-15)y_{n+1}&=(128n^3+40n^2-82n-45)y_n \\
&-n^2(256n^3-240n^2+64n-7)y_{n-1}
+n^2(n-1)^2(16n+1)y_{n-2}
\label{eq04}
\end{split}
\end{equation}
with the initial conditions
\begin{equation*}
\begin{array}{ccc}
\tilde{p}_0=0, \qquad & \qquad \tilde{p}_1=2, \qquad & \qquad \tilde{p}_2=31, \\
\tilde{q}_0=1,  \qquad &  \qquad \tilde{q}_1=3, \qquad &  \qquad \tilde{q}_2=50
\end{array}
\end{equation*}
and  the  asymptotics
\begin{align}
\tilde{q}_n&=(2n)!\frac{e^{\sqrt{2n}}}{\sqrt[4]{n}}\left(\frac{1}{\sqrt{\pi}(4e)^{3/8}}+O(n^{-1/2})\right),\nonumber\\[3pt]
\tilde{p}_n-\gamma \tilde{q}_n&=(2n)!\frac{e^{-\sqrt{2n}}}{\sqrt[4]{n}}\left(\frac{2\sqrt{\pi}}{(4e)^{3/8}}+O(n^{-1/2})\right).
\label{eq05}
\end{align}
The remainder of the above approximations is given by the integral \cite{at}
\begin{equation}
\int_0^{\infty}Q_n(x)e^{-x}\log(x)\,dx=\tilde{p}_n-\gamma \tilde{q}_n,
\label{eq06}
\end{equation}
where
$$
Q_n(x)=\frac{1}{n!^2}\frac{e^x}{1-x}(x^n(x^n(1-x)^{2n+1}e^{-x})^{(n)})^{(n)}
$$
is a multiple Jacobi-Laguerre orthogonal polynomial
on $[0,1]$
and $[1,+\infty)$
with respect to the two weight functions
$w_1(x)=(1-x)e^{-x}, \, w_2(x)=(1-x)\log(x)e^{-x}.$
The integral  (\ref{eq06}) can also be written as a multiple integral (see \cite[Lemma 4]{he})
\begin{equation}
\int_0^{\infty}Q_n(x)e^{-x}\log(x)\,dx=\int_0^{\infty}\int_0^{\infty}
\frac{x^ny^n(x-1)^{2n+1}e^{-x}}{(xy+1)^{n+1}(y+1)^{n+1}}\,dxdy.
\label{eq07}
\end{equation}
The  integrality of the sequences $\tilde{p}_n$ and $\tilde{q}_n$  is not evident and can not be deduced
directly from the recurrence equation (\ref{eq04}).
Tulyakov \cite{tu} proved independently that $\tilde{p}_n$ and $\tilde{q}_n$ are integers,
by considering a more ''dense`` sequence of rational approximations to $\gamma.$
The present authors (see \cite{he1}) found explicit representations for $\tilde{p}_n$ and $\tilde{q}_n:$
\begin{equation}
\tilde{q}_n=\sum_{k=0}^n\binom{n}{k}^2(n+k)!, \qquad
\tilde{p}_n=\sum_{k=0}^n\binom{n}{k}^2(n+k)!(H_{n+k}+2H_{n-k}-2H_k),
\label{eq08}
\end{equation}
where $H_n=\sum_{k=1}^n1/k$ is the $n$th harmonic number and $H_0:=0.$
Formulae (\ref{eq08}) imply that $\tilde{q}_n$ and $\tilde{p}_n$ are integers divisible by $n!$
and $\frac{n!}{D_n},$ respectively. Although the coefficients of the linear forms
(\ref{eq05}) can be canceled out by the large common factor $\frac{n!}{D_n},$
it is still not enough to prove the irrationality of $\gamma,$
since the linear $\gamma$-forms with integer coefficients:
$$
\frac{\tilde{p}_nD_n}{n!}-\frac{\tilde{q}_nD_n}{n!}\,\gamma
\in{\mathbb Z}+{\mathbb Z}\gamma
$$
do not tend to zero as $n$ tends to infinity:
$$
\frac{\tilde{p}_nD_n}{n!}-\frac{\tilde{q}_nD_n}{n!}\,\gamma=O(4^nn^{n-1/4}e^{-\sqrt{2n}}).
$$
 Nevertheless,
the sequence $\tilde{p}_n/\tilde{q}_n$ provides good rational approximations to Euler's constant
$$
\frac{\tilde{p}_n}{\tilde{q}_n}-\gamma=2\pi e^{-2\sqrt{2n}}(1+O(n^{-1/2})) \qquad\mbox{as}
\quad n\to\infty.
$$
In 2009, T.~Rivoal \cite{ri1} found another way of rationally  approximating
the Euler constant $\gamma,$ by using multiple Laguerre polynomials
$$
A_n(x)=\frac{1}{n!^2}e^x(x^n(x^ne^{-x})^{(n)})^{(n)}.
$$
His construction is based on the following third-order recurrence:
\begin{equation*}
\begin{split}
(n&+3)^2(8n+11)(8n+19)y_{n+3}=(n+3)(8n+11)(24n^2+145n+215)y_{n+2}\\
&-(8n+27)(24n^3+105n^2+124n+25)y_{n+1}+(n+2)^2(8n+19)(8n+27)y_n,
\end{split}
\end{equation*}
which provides two sequences of rational numbers $P_n$ and
$Q_n,$ $n\ge 0,$ with the initial values
\begin{equation*}
\begin{array}{lcc}
P_0=-1, \qquad & \qquad P_1=4, \qquad & \qquad P_2=77/4, \\
Q_0=1,  \qquad &  \qquad Q_1=7, \qquad &  \qquad Q_2=65/2
\end{array}
\end{equation*}
such that $\frac{P_n}{Q_n}$ converges to $\gamma.$
The sequences $P_n,$ $Q_n$ satisfy the inclusions
$$
n!\,Q_n, \quad n!\,D_n P_n\in {\mathbb Z},
$$
which were proved in  \cite[Corollary 5]{he},
and provide better approximations to $\gamma,$
$$
\left|\frac{P_n}{Q_n}-\gamma\right|\le c_0e^{-9/2n^{2/3}+3/2n^{1/3}},
\quad |Q_n|=O(e^{3n^{2/3}-n^{1/3}}) \quad\mbox{as}\quad n\to\infty.
$$
Unfortunately, this convergence
is not fast enough to imply the irrationality of $\gamma.$

In this paper, we give a new interpretation of Aptekarev's approximations to $\gamma$
in terms of Meijer $G$-functions and hypergeometric-type series, which can be considered as an analog
of the complex integral  (\ref{eq02}) and series (\ref{eq01}) for Euler's constant.
This approach allows us to describe a very general construction giving linear forms in $1$ and $\gamma$
with rational coefficients. Using this construction we find new rational approximations to $\gamma$
generated by a second-order inhomogeneous linear recurrence with polynomial coefficients. This leads
to a continued fraction (though not a simple continued fraction) for Euler's constant.
It seems to be the first non-trivial continued fraction expansion convergent to Euler's constant sub-exponentially,
 the elements  of which can be expressed as a general pattern. It is interesting to note that the same
 homogeneous second-order linear recurrence generates a continued fraction for the Euler-Gompertz constant
 found by Stieltjes in 1895.

\section{Analogs of hypergeometric-type series and complex integrals for Euler's constant}

Note that the first attempt to generalize series (\ref{eq01}) to find suitable approximations
for Euler's constant $\gamma$ was made by Sondow \cite{sondow}. He introduced the following series:
$$
I_n:=\sum_{\nu=n+1}^{\infty}\int_{\nu}^{\infty}\left(\frac{n!}{x(x+1)\cdots (x+n)}\right)^2\,dx
$$
and proved that
\begin{equation}
I_n=\binom{2n}{n}\gamma+L_n
-\sum_{i=0}^n\binom{n}{i}^2H_{n+i}=O(2^{-4n}n^{-1}),
\label{sond}
\end{equation}
where
$$
L_n=2\sum\limits_{k=1}^n\sum\limits_{i=0}^{k-1}\binom{n}{i}^2(H_{n-i}-H_i)\log(n+k).
$$
Unfortunately, this type of approximations  contains linear forms in logarithms of rational numbers and this enables
only to obtain conditional irrationality criteria for $\gamma$ that require knowledge of the growth
of the fractional parts $\{D_{2n}L_n\}.$

\vspace{0.1cm}

In this section, we obtain new representations for Aptekarev's linear form $\tilde{p}_n-\gamma \tilde{q}_n$
distinct from (\ref{eq06}) and (\ref{eq07}) which can be considered as analogs of the hypergeometric-type series
(\ref{eq01}) and complex integral (\ref{eq02}). It can be easily done by using explicit formulae
(\ref{eq08}).

\begin{proposition} \label{p1}
For each $n=0,1,2,\ldots,$ the following equality holds:
$$
\tilde{f}_n:=\tilde{p}_n-\gamma \tilde{q}_n=n!^2\sum_{k=0}^n\frac{d}{dt}\left.\left(\frac{\Gamma(n+t+1)}{\Gamma^2(t+1)
\Gamma^2(n-t+1)}\right)\right|_{t=k}.
$$
\end{proposition}
\begin{proof}
The straightforward verification shows that
\begin{equation}
\begin{split}
\frac{d}{dt}\left(\frac{\Gamma(n+t+1)}{\Gamma^2(t+1)
\Gamma^2(n-t+1)}\right)&=\frac{\Gamma(n+t+1)}{\Gamma^2(t+1)
\Gamma^2(n-t+1)} \\ &\times\left(\psi(n+t+1)-2\psi(t+1)+2\psi(n-t+1)\right),
\end{split}
\label{eq09}
\end{equation}
where $\psi(z)=\frac{\Gamma'(z)}{\Gamma(z)}$ is the logarithmic derivative of the gamma function,
also known as the digamma function. Summing (\ref{eq09}) for $t=0,1,\ldots,n$ and using the
well-known properties of the digamma function
$$
\psi(1)=-\gamma, \qquad \psi(n+1)=H_n-\gamma, \quad n\ge 1,
$$
we have
\begin{equation*}
\begin{split}
&n!^2\sum_{k=0}^n\frac{d}{dt}\left.\left(\frac{\Gamma(n+t+1)}{\Gamma^2(t+1)
\Gamma^2(n-t+1)}\right)\right|_{t=k} \\ &=
n!^2\sum_{k=0}^n\frac{\Gamma(n+k+1)}{\Gamma^2(k+1)\Gamma^2(n-k+1)}
\left(\psi(n+k+1)-2\psi(k+1)+2\psi(n-k+1)\right) \\
&=\sum_{k=0}^n\binom{n}{k}^2(n+k)!(H_{n+k}-2H_k+2H_{n-k}-\gamma)
=\tilde{p}_n-\gamma \tilde{q}_n,
\end{split}
\end{equation*}
as required.
\end{proof}
\begin{proposition} \label{p2}
For each $n=0,1,2,\ldots,$ we have
\begin{equation}
\begin{split}
\tilde{f}_n&+\frac{n!^4}{(2n+1)!^2}\,{}\sb 2F\sb{2}
\left(
\left.\atop{
n+1, n+1}{2n+2, 2n+2}\right|-1
\right)\\[3pt]
&=\frac{n!^2}{2\pi i}\int_{c-i\infty}^{c+i\infty}\frac{\Gamma(n+t+1)}{\Gamma^2(t+1)
\Gamma^2(n-t+1)}\left(\frac{\pi}{\sin\pi t}\right)^2dt \\[5pt]
&=\frac{n!^2}{2\pi i}\int_{c-i\infty}^{c+i\infty}\frac{\Gamma(n+t+1)\Gamma^2(t-n)}{\Gamma^2(t+1)}\,dt
=n!^2 G_{3,2}^{0,3}\left(
\left.\atop{-n,
n+1, n+1}{0, \,  0}\right|1
\right),
\label{eq10}
\end{split}
\end{equation}
\begin{equation}
\begin{split}
\tilde{q}_n&=\frac{n!^2}{2\pi i}\int_L\frac{\Gamma(n+t+1)\,e^{i\pi t}}%
{\Gamma^2(t+1)\Gamma^2(n-t+1)}\cdot\frac{\pi}{\sin\pi t}\,dt \\[5pt]
&=\frac{(-1)^n n!^2}{2\pi i}\int_L\frac{\Gamma(n+t+1)\Gamma(t-n)}%
{\Gamma^2(t+1)\Gamma(n-t+1)}\,e^{i\pi t}\,dt \\[5pt]
&=(-1)^n n!^2 G_{3,2}^{0,2}\left(
\left.\atop{-n,
n+1, n+1}{0, \,  0}\right|-1
\right),
\label{eq10.5}
\end{split}
\end{equation}
where ${}\sb rF\sb{s}$ is the generalized hypergeometric function,
$c$ is an arbitrary real number satisfying $c>n,$ and $L$ is a loop beginning and
ending at $-\infty$ and encircling the points $n,n-1,n-2, \ldots$ exactly once in the
positive direction.
\end{proposition}
\begin{proof}
We first note that the equality of two complex integrals in (\ref{eq10})
and (\ref{eq10.5})  follows easily by the reflection
formula for the gamma function:
$$
\Gamma(t-n)\Gamma(n-t+1)=\frac{(-1)^n \pi}{\sin\pi t}.
$$
The third equality in both formulae  (\ref{eq10})
and (\ref{eq10.5})  follows by the definition of the Meijer $G$-function.
To prove the first equality in (\ref{eq10}),
we  consider the integrands of the complex integrals (\ref{eq10}) on the rectangular contour
with vertices $c\pm iN,$ $-N-1/2\pm iN,$ where $N$ is a sufficiently large integer, $N>c.$
Then, by the residue theorem, we have that the integral
\begin{equation}
\frac{1}{2\pi i}\left(\int_{c-iN}^{c+iN}+\int_{c+iN}^{-N-\frac{1}{2}+iN}
+\int_{-N-\frac{1}{2}+iN}^{-N-\frac{1}{2}-iN}+\int_{-N-\frac{1}{2}-iN}^{c-iN}\right)
\frac{\Gamma(n+t+1)\Gamma^2(t-n)}{\Gamma^2(t+1)}\,dt
\label{eq11}
\end{equation}
is equal to the sum of the residues of the integrand at integer points $t=k,$ $-N\le k\le n.$
Using the expansion
$$
\left(\frac{\pi}{\sin\pi t}\right)^2=\frac{1}{(t-k)^2}+O(1)
$$
in a neighborhood of the integer point $t=k$ we obtain that the integral (\ref{eq11})
is equal to
\begin{equation}
\begin{split}
\sum_{k=0}^n\,&\underset{t=k}{\rm res}\left(\frac{\Gamma(n+t+1)}{\Gamma^2(t+1)\Gamma^2(n-t+1)}
\left(\frac{\pi}{\sin\pi t}\right)^2\right)+\sum_{k=-N}^{-n-1}
\underset{t=k}{\rm res}\left(\frac{\Gamma(n+t+1)\Gamma^2(t-n)}{\Gamma^2(t+1)}\right) \\[3pt]
&=\sum_{k=0}^n
\frac{d}{dt}\left.\left(\frac{\Gamma(n+t+1)}{\Gamma^2(t+1)
\Gamma^2(n-t+1)}\right)\right|_{t=k}+\sum_{k=0}^{N-n-1}\frac{(n+k)!^2(-1)^k}{(2n+k+1)!^2k!}.
\label{eq12}
\end{split}
\end{equation}
Since on the sides $[c+iN,-N-1/2+iN],$ $[-N-1/2+iN,-N-1/2-iN],$ $[-N-1/2-iN,c-iN]$
of the rectangle we have $|t|=O(N),$ it follows that
\begin{equation}
\left|\frac{\Gamma^2(t-n)}{\Gamma^2(t+1)}\right|=\frac{1}{|t^2(t-1)^2\cdots (t-n)^2|}
=O\left(\frac{1}{N^{2n+2}}\right).
\label{eq13}
\end{equation}
For large $|z|$ the asymptotic expansion  of the gamma function is \cite[Section 2.11]{lu}
\begin{equation}
\log\Gamma(z)=\left(z-\frac{1}{2}\right)\log z-z+\frac{1}{2}\log(2\pi)+O(|z|^{-1}),
\label{eq14}
\end{equation}
where $|\arg z|\le \pi-\varepsilon,$
$\varepsilon>0$ and the constant in $O$ is independent of $z.$
Then  for $t=x\pm iN,$ $-N\le x\le c,$ we have
\begin{equation}
|\Gamma(n+t+1)|=|\Gamma(x+n+1\pm iN)|=O(e^{(x+n+\frac{1}{2})\log N\mp N\arg(x+n+1\pm iN)})
\le O(N^{c+n+\frac{1}{2}}e^{-\frac{\pi}{4}N}).
\label{eq15}
\end{equation}
On the segment $[-N-1/2+iN,-N-1/2-iN]$ we use the trivial estimate
\begin{equation}
|\Gamma(n+t+1)|\le |\Gamma({\rm Re}\,(n+t+1))|=|\Gamma(n+1/2-N)|=
\frac{\pi}{\Gamma(N+1/2-n)}=O(e^{-N\log N+N}).
\label{eq16}
\end{equation}
Summarizing (\ref{eq13}), (\ref{eq15}), (\ref{eq16}) and letting $N$ tend to infinity in
(\ref{eq11}), by (\ref{eq12}) and Proposition \ref{p1}, we get the first equality
in (\ref{eq10}).

Now we prove the first equality in (\ref{eq10.5}). For this purpose, suppose that
$C_1$ and $C_2$ are points of intersections of the loop $L$ with the vertical line
${\rm Re}\,(t)=-N-1/2$ and consider a closed contour $L^*$ oriented in the positive direction
and consisting of the segment $C_1C_2$ and the right part of the loop $L$ connecting the points
$C_1$ and $C_2,$ which we denote by $\widetilde{C_1C_2}.$
Then, by the residue theorem, we obtain
\begin{equation}
\frac{n!^2}{2 i}\!\int_{L^*}\!\frac{\Gamma(n+t+1)\,e^{i\pi t}}%
{\Gamma^2(t+1)\Gamma^2(n-t+1)}\,\frac{dt}{\sin\pi t}=
n!^2\!\sum_{k=0}^n\underset{t=k}{\rm res}\left(
\frac{\Gamma(n+t+1)\,e^{i\pi t}}%
{\Gamma^2(t+1)\Gamma^2(n-t+1)}\,\frac{\pi}{\sin\pi t}\right)=\tilde{q}_n.
\label{eq16.5}
\end{equation}
On the other hand, we have
\begin{equation}
\int_{L^*}=\int_{\widetilde{C_1C_2}}+\int_{C_2C_1}.
\label{eq17}
\end{equation}
Since
$$
\frac{\Gamma(n+t+1)\,e^{i\pi t}}{\Gamma^2(t+1)\Gamma^2(n-t+1)}\,
\frac{\pi}{\sin\pi t}=\frac{1}{2\pi i}\,\frac{\Gamma^2(t-n)}{\Gamma^2(t+1)}
\Gamma(n+t+1)(e^{2i\pi t}-1)
$$
and the function $e^{2i\pi t}-1$ is bounded on the vertical segment $C_1C_2$,
by (\ref{eq13}), (\ref{eq16}) and (\ref{eq17}) we get
\begin{equation*}
\int_{L^*}=\int_{\widetilde{C_1C_2}}+O(e^{-N\log N+N}).
\end{equation*}
Finally, letting $N$ tend to infinity
and taking into account  (\ref{eq16.5}), we get the required formula.
\end{proof}

%\vspace{0.2cm}

Now, more generally,  consider an arbitrary function $F(n,t)$ of the form
\begin{equation}
F(n,t)=\frac{\prod_{j=1}^s\Gamma(a_jn+b_jt+1)}{\prod_{j=1}^u\Gamma(c_jn+d_jt+1)},
\label{fnt}
\end{equation}
where $a_j, b_j, c_j, d_j\in {\mathbb Z},$ $\sum_{j=1}^sb_j\ne\sum_{j=1}^ud_j,$
and all gamma values   are well defined for $0\le t\le M(n).$
We say that $\Gamma(an+bt+1)$ is well defined if $an+bt+1$ is not a negative integer or zero.
\begin{proposition} \label{p3}
Let $F(n,t)$ be defined as above, $M(n)$ be a non-negative integer, and
$$
F_n=\sum_{t=0}^{M(n)}\frac{d}{dt}F(n,t).
$$
Then for each $n=0,1,2,\ldots,$ we have $F_n=p_n-\gamma q_n$ with
$$
q_n=\Bigl(\sum_{j=1}^sb_j-\sum_{j=1}^ud_j\Bigr)\cdot
\sum_{k=0}^{M(n)}F(n,k) %\frac{\prod_{j=1}^s(a_jn+b_jk)!}{\prod_{j=1}^u(c_jn+d_jk)!}
$$
and
$$
p_n=\sum_{k=0}^{M(n)}F(n,k)
\Bigl(\sum_{j=1}^sb_jH_{a_jn+b_jk}-\sum_{j=1}^ud_jH_{c_jn+d_jk}\Bigr).
$$
\end{proposition}
\begin{proof} Differentiating $F(n,t)$ with respect to $t$ and summing over $t=0,1,\ldots,M(n),$
we have
\begin{equation*}
\begin{split}
\sum_{t=0}^{M(n)}&\frac{d}{dt}F(n,t)=\sum_{k=0}^{M(n)}
\frac{\prod_{j=1}^s\Gamma(a_jn+b_jk+1)}{\prod_{j=1}^u\Gamma(c_jn+d_jk+1)}\\
&\times\Bigl(\sum_{j=1}^sb_j\psi(a_jn+b_jk+1)-\sum_{j=1}^ud_j\psi(c_jn+d_jk+1)\Bigr)\\
&\!=\sum_{k=0}^{M(n)}
\frac{\prod_{j=1}^s(a_jn+b_jk)!}{\prod_{j=1}^u(c_jn+d_jk)!}
\Bigl(\sum_{j=1}^sb_j(H_{a_jn+b_jk}-\gamma)-
\sum_{j=1}^ud_j(H_{c_jn+d_jk}-\gamma)\Bigr)
=p_n-\gamma q_n.
\end{split}
\end{equation*}
\end{proof}

\section{A second-order inhomogeneous linear recurrence for Euler's constant}

In this section, we consider application of Proposition \ref{p3} to the function
$$
F(n,t)=\frac{\Gamma^2(n+1)}{\Gamma(t+1)\Gamma^2(n-t+1)}, \quad\qquad n\in{\mathbb Z}, \,\,\,n\ge 0.
$$
Then for
$$
F_n:=\sum_{t=0}^n\frac{d}{dt}F(n,t),
$$
we have $F_n=p_n-\gamma q_n$ with
\begin{equation}
q_n=\sum_{k=0}^n\binom{n}{k}^2 k!, \qquad
p_n=\sum_{k=0}^n\binom{n}{k}^2 k! (2H_{n-k}-H_k), \qquad n=0,1,2,\ldots.
\label{eq99}
\end{equation}
\begin{lemma} \label{l1}
The sequence $\{q_n\}_{n=0}^{\infty}$ is a solution of the second-order homogeneous
linear recurrence
\begin{equation}
q_{n+2}-2(n+2)q_{n+1}+(n+1)^2q_n=0
\label{eq101}
\end{equation}
with the initial values $q_0=1,$ $q_1=2,$
and the sequence $\{p_n\}_{n=0}^{\infty}$ is a solution of the second-order inhomogeneous
linear recurrence
\begin{equation}
p_{n+2}-2(n+2)p_{n+1}+(n+1)^2p_n=-\frac{n}{n+2}
\label{eq102}
\end{equation}
with the initial values $p_0=0,$ $p_1=1.$
\end{lemma}
\begin{proof}
Applying Zeilberger's algorithm of creative telescoping \cite[Chapter 6]{PWZ} to the function $F(n,t)$
we get for each $n=0,1,2,\ldots$ the identity
\begin{equation}
F(n+2,t)-2(n+2)F(n+1,t)+(n+1)^2F(n,t)=G(n,t+1)-G(n,t),
\label{eq103}
\end{equation}
where
$$
G(n,t)=\frac{\Gamma^2(n+2)\cdot r(n,t)}{\Gamma(t+1) \Gamma^2(n-t+3)}\qquad\mbox{and}\qquad
r(n,t)=t(t^2-(2n+3)t+n(n+2)).
$$
To prove (\ref{eq103}) it is sufficient to multiply both sides of (\ref{eq103}) by
$\Gamma^2(n-t+3) \Gamma(t+1)/\Gamma^2(n+2)$ and after cancelation of gamma factors
to verify the identity
$$
(n+2)^2-2(n+2)(n-t+2)^2+(n-t+2)^2(n-t+1)^2=\frac{(n-t+2)^2}{t+1}r(n,t+1)-r(n,t).
$$
Summing equality (\ref{eq103}) over $t=0,1,2,\ldots$ and taking into account that
$$
\lim_{t\to k}G(n,t)=0, \qquad k=n+3, n+4, \ldots,
$$
we get the difference equation for $q_n:$
$$
q_{n+2}-2(n+2)q_{n+1}+(n+1)^2q_n=-G(n,0)=0.
$$
In order to get the recurrence relation for the sequence $p_n,$
it is convenient to rewrite $F_n$ as an infinite sum
$$
F_n:=\sum_{t=0}^{\infty}\frac{d}{dt}F(n,t),
$$
taking into account that
$$
\lim_{t\to k}\frac{d}{dt}F(n,t)=0 \qquad\mbox{for} \qquad k=n+1, n+2, \ldots.
$$
Then differentiating (\ref{eq103}) with respect to $t$ and summing over $t=0,1,2,\ldots$
we get for each $n=0,1,2,\ldots,$
\begin{equation}
F_{n+2}-2(n+2)F_{n+1}+(n+1)^2F_n=\lim_{k\to\infty}G'(n,k+1)-G'(n,0).
\label{eq104}
\end{equation}
Since
\begin{equation*}
\begin{split}
G'(n,t)=(n+1)!^2&\left(\frac{t^3-(2n+3)t^2+tn(n+2)}{\Gamma(t+1)\Gamma^2(n-t+3)}
(2\psi(n-t+3)-\psi(t+1))\right. \\[3pt]
&\left.+\frac{3t^2-2t(2n+3)+n(n+2)}{\Gamma(t+1)\Gamma^2(n-t+3)}\right),
\end{split}
\end{equation*}
we see that
$$
\lim_{k\to\infty}G'(n,k+1)=0 \qquad\mbox{and}\qquad G'(n,0)=\frac{n}{n+2},
$$
and consequently, (\ref{eq104}) becomes
$$
F_{n+2}-2(n+2)F_{n+1}+(n+1)^2F_n=-\frac{n}{n+2}, \qquad n=0,1,2,\ldots.
$$
This implies that the sequence $p_n=F_n+\gamma q_n$ satisfies the same inhomogeneous recurrence,
and the lemma is proved.
\end{proof}

\section{Rate of convergence of rational approximations}

In this section we show that the sequence $p_n/q_n$ converges to Euler's constant $\gamma$ and investigate its rate of convergence.
We begin with defining  a complex integral $I_n$ by means of
the Meijer $G$-function:
\begin{equation}
I_n:=n!^2\,G_{2,1}^{0,2}\left(
\left.\atop{n+1,
n+1}{0}\right|1
\right)=\frac{n!^2}{2\pi i}\int_{c-i\infty}^{c+i\infty}
\frac{\Gamma^2(t-n)}{\Gamma(t+1)}\,dt,
\label{eq18}
\end{equation}
where $c>n$ is an arbitrary
constant.
\begin{lemma} \label{l2}
The following formula holds:
\begin{equation*}
 F_n=I_n+O\left(\frac{1}{n^2}\right) \qquad \text{as}\qquad n\to\infty,
\end{equation*}
where the constant in $O$ is absolute.
\end{lemma}
\begin{proof} Since
\begin{equation}
\frac{\Gamma^2(t-n)}{\Gamma(t+1)}=\Gamma(t-n)
\left(\frac{\Gamma(t-n)}{\Gamma(t+1)}\right),
\label{eq19}
\end{equation}
by a similar argument as in the proof of Proposition \ref{p2}, considering
the integrand (\ref{eq19}) on the rectangular contour with vertices $c\pm iN,$
$-N-1/2\pm iN,$ where $N$ is a sufficiently large integer, we conclude
that the integral (\ref{eq18}) can be evaluated as a sum of residues
at the points $n, n-1, \ldots.$ It is easily seen that the function
(\ref{eq19}) has double poles at the points
$0,1,2,\ldots,n$ and simple poles at $-1,-2,\ldots.$ Therefore, we have
\begin{equation*}
\begin{split}
I_n&=n!^2 \sum_{k=-\infty}^n\underset{t=k}{\rm res}\left(\frac{\Gamma^2(t-n)}{\Gamma(t+1)}
\right) \\
&=
n!^2 \sum_{k=-\infty}^{-1}\underset{t=k}{\rm res}\left(
\frac{\pi}{\sin\pi t}\cdot\frac{1}{\Gamma(n-t+1)\cdot t(1-t)\cdots(n-t)}\right) \\
&+n!^2 \sum_{k=0}^n\underset{t=k}{\rm res}\left(\frac{1}{\Gamma(t+1)%
\Gamma^2(n-t+1)}\left(\frac{\pi}{\sin\pi t}\right)^2\right)
 \\
&=n!^2\sum_{k=-\infty}^{-1}\frac{(-1)^k}{(n-k)!\cdot k(1-k)\cdots(n-k)} \\
&+
n!^2\sum_{k=0}^n\frac{d}{dt}\left.\left(\frac{1}{\Gamma(t+1)\Gamma^2(n-t+1)}\right)\right|_{t=k}
=\frac{1}{(n+1)^2}\sum_{k=0}^{\infty}\frac{(-1)^k k!}{(n+2)_k^2}+F_n.
\end{split}
\end{equation*}
Since
$$
\left|\sum_{k=0}^{\infty}\frac{(-1)^k k!}{(n+2)_k^2}\right|\le\sum_{k=0}^{\infty}\frac{k!}{(n+2)^2_k}
\le\sum_{k=0}^{\infty}\frac{1}{k!}=e,
$$
we get the desired assertion.
\end{proof}
\begin{lemma}
The following asymptotic formulae hold
\begin{equation*}
\begin{split}
F_n&=n!\, \frac{e^{-2\sqrt{n}}}{n^{1/4}}\left(\sqrt{\frac{\pi}{e}}+O(n^{-1/2})\right) \qquad\text{as}\qquad n\to\infty, \\[3pt]
q_n&=n!\, \frac{e^{2\sqrt{n}}}{n^{1/4}}\left(\frac{1}{2\sqrt{\pi e}}+O(n^{-1/2})\right) \qquad\text{as}\qquad n\to\infty.
\end{split}
\end{equation*}
\end{lemma}
\begin{proof}
It is easy to show that the complex integral $I_n$ can be expressed in terms of the Whittaker function $W_{\kappa, \mu}(z)$
which is one of the solutions of Whittaker's confluent hypergeometric equation (see \cite[Section 2]{buch}, \cite[Section 1]{slater}):
$$
\frac{d^2 y}{dz^2}+\left(-\frac{1}{4}+\frac{\kappa}{z}+\frac{1-4\mu^2}{4z^2}\right)y=0.
$$
Indeed, applying the following transformation for the Meijer G-functions (see \cite[Section 5.3]{lu}):
$$
G_{p,q}^{m,n}\left(
\left.\atop{a_1,\ldots,a_p
}{b_1, \ldots, b_q}\right|z
\right)=G_{q,p}^{n,m}\left(
\left.\atop{1-b_1,\ldots,1-b_q
}{1-a_1, \ldots, 1-a_p}\right|z^{-1}
\right)
$$
and taking into account (see \cite[Section 6.4, (6)]{lu}) that
$$
z^{1/2}G_{1,2}^{2,0}\left(
\left.\atop{a+1/2-\kappa
}{a+\mu, a-\mu}\right|z
\right)=z^a e^{-z/2} W_{\kappa, \mu}(z),
$$
we obtain
\begin{equation} \label{eq119}
I_n=n!^2 e^{-1/2} W_{-n-1/2, 0}(1).
\end{equation}
The asymptotic behavior of the Whittaker function $W_{\kappa, \mu}(z)$ for various conditions on parameters is well investigated (see, for example,
\cite[Chapter 4]{slater}, \cite[Chapter 3]{buch}). From \cite[Section 7.4, (20)]{buch} we easily find that
$$
W_{-n-1/2, 0}(1)=\frac{e^{n-2\sqrt{n}}}{\sqrt{2}\, n^{n+3/4}}\,(1+O(n^{-1/2})) \qquad\text{as}\,\,\,n\to\infty,
$$
which by (\ref{eq119}) and Lemma \ref{l2}, implies the asymptotic formula for $F_n.$

To compute the asymptotics of $q_n,$ we note that $q_n=n! \mathcal{L}_n(-1),$
where $\mathcal{L}_n(x)=\frac{1}{n!}(x^ne^{-x})^{(n)}=\sum_{k=0}^n\binom{n}{k}
\frac{(-x)^k}{k!}$ is the Laguerre polynomial. Then the Perron asymptotic formula  for the
confluent hypergeometric function ${}\sb 1F\sb{1}(a\pm n;b;z)$ \cite{perron} yields (see \cite[p.~199]{szego})
$$
q_n=n!\, \frac{e^{2\sqrt{n}}}{\sqrt[4]{n}}\left(\frac{1}{2\sqrt{\pi e}}+O(n^{-1/2})\right),
$$
and the lemma is proved.
\end{proof}
\begin{theorem}
Let $\{q_n\}_{n\ge 0},$ $\{p_n\}_{n\ge 0}$ be defined by {\rm(\ref{eq99})}. Then
$q_n\in {\mathbb Z},$ $D_np_n\in {\mathbb Z}$ for each $n=0,1,2,\ldots,$ and
$$
p_n-\gamma q_n=n!\, \frac{e^{-2\sqrt{n}}}{\sqrt[4]{n}} \left(
\sqrt{\frac{\pi}{e}}+O(n^{-1/2})\right), \qquad
q_n=n!\, \frac{e^{2\sqrt{n}}}{\sqrt[4]{n}}\left(\frac{1}{2\sqrt{\pi e}}+O(n^{-1/2})\right)
$$
as $n\to\infty.$
\end{theorem}

\begin{corollary}
The sequence $p_n/q_n$ converges to Euler's constant sub-exponentially:
$$
\frac{p_n}{q_n}-\gamma=e^{-4\sqrt{n}}(2\pi+O(n^{-1/2}))\qquad\text{as}\qquad
n\to\infty.
$$
\end{corollary}
It is interesting to mention that the sequence of complex integrals $I_n$ produces good rational approximations to
the Euler-Gompertz constant
\begin{equation*}
\delta:=\int_0^{\infty}\frac{e^{-x}}{x+1}\,dx=
\int_0^{\infty}\log(x+1)e^{-x}\,dx=0.5963473623\ldots.
\end{equation*}
\begin{theorem} \label{t2}
For each $n=0,1,2,\ldots,$
$$
eI_n=q_n\delta-s_n,
$$
where $s_n$ is a solution of recurrence
{\rm(\ref{eq101})} with the initial values $s_0=0,$ $s_1=1.$ In particular, $s_n/q_n$ converges to the Euler-Gompertz constant sub-exponentially:
$$
\delta-\frac{s_n}{q_n}=e^{-4\sqrt{n}}(2\pi e+O(n^{-1/2})) \qquad\text{as}\qquad n\to\infty.
$$
\end{theorem}
\begin{proof}
Whittaker's function $W_{\kappa, \mu}(z)$ satisfies the second-order recurrence relation (see \cite[Section 2.5.1]{slater}):
\begin{equation} \label{recurrence}
(2\kappa-z)W_{\kappa, \mu}(z)+W_{\kappa+1,\mu}(z)-(\mu-\kappa+1/2)(\mu+\kappa-1/2)W_{\kappa-1,\mu}(z)=0.
\end{equation}
Then by (\ref{eq119}), we easily conclude that the sequence of integrals $I_n$ satisfies the recurrence~(\ref{eq101}).
From \cite[Section 5.6]{slater} and \cite[Section 2.6]{buch} we have that
$$
W_{-1/2,0}(1)=e^{1/2}\int_1^{\infty}\frac{e^{-t}}{t}\,dt=e^{-1/2}\delta, \qquad
W_{1/2,0}(1)=e^{-1/2}.
$$
Then from (\ref{recurrence}) we easily find that
$$
W_{-3/2,0}(1)=2e^{-1/2}\delta-e^{-1/2}.
$$
For the first several values of the sequence $I_n$ we have
\begin{equation*}
\begin{split}
eI_0&=e^{1/2}W_{-1/2,0}(1)=\delta=q_0\delta-s_0, \\
eI_1&=e^{1/2}W_{-3/2,0}(1)=2\delta-1=q_1\delta-s_1,
\end{split}
\end{equation*}
with $s_0=0,$ $s_1=1$ and $q_0,$ $q_1$ defined in Lemma \ref{l1}.
Since the sequence $\{eI_n\}_{n\ge 0}$ satisfies the recurrence equation (\ref{eq101}), we easily obtain that for any $n\ge 0,$
$$
eI_n=q_n\delta-s_n,
$$
which completes the proof.
\end{proof}
From Theorem \ref{t2} we recover a continued fraction expansion for the Euler-Gompertz constant that was first proved
by Stieltjes in 1895 (see \cite[Chapter 18, (92.7)]{wall}).
\begin{corollary}
The Euler-Gompertz constant has the following continued fraction expansion:
\begin{equation*}
\delta=\frac{1}{2} \cf{}{\,+}\, \,\underset{m=1}{\overset{\infty}{\bf K}}\left(\frac{-m^2}{2(m+1)}\right)
=\frac{1}{2}\cf{}{-}\frac{1^2}{4}\cf{}{-}\frac{2^2}{6}\cf{}{-}
\frac{3^2}{8}\cf{}{-}\ldots.
\end{equation*}
\end{corollary}
The $n$th convergent of this continued fraction has the form $s_n/q_n$ and it
slowly converges  to $\delta$ to imply certain
results on arithmetical nature of $\delta.$
The irrationality of the Euler-Gompertz constant $\delta$ is still an open problem.
Using the representation of
$\delta=e E_1(1)$
in terms of the exponential integral
$$
E_1(z)
=\int_{z}^{\infty}\frac{e^{-t}}{t}\,dt, \qquad \qquad |\arg z|<\pi,
$$
that can be expanded in powers of $z$ as
(see \cite[p.\ 228]{abr}),
$$
E_1(z)=-\gamma-\log z-\sum_{k=1}^{\infty}\frac{(-1)^kz^k}{k! k}, \qquad |\arg z|<\pi,
$$
we obtain the following relation connecting three famous constants $e,$ $\gamma,$ and $\delta,$
\begin{equation} \label{eq120}
-\delta=e\gamma+e\sum_{k=1}^{\infty}\frac{(-1)^k}{k! k}.
\end{equation}
From classical results of Shidlovskii on algebraic independence of values of E-functions, it follows
(see \cite[Ch.~7, Th.~1]{shidlovski})
that the numbers $e$ and $\sum_{k=1}^{\infty}\frac{(-1)^k}{k! k}$ are algebraically
independent over ${\mathbb Q}.$ Then from relation (\ref{eq120}) we obtain
that the numbers $e$ and $\delta/e+\gamma$ are algebraically independent over ${\mathbb Q}$ and therefore
we arrive at the following result due to Mahler~\cite{Mah}.
\begin{corollary}
At least one of the numbers $\gamma, \delta$ must be transcendental.
\end{corollary}

\section{A continued fraction for Euler's constant}

Based on the results obtained in the previous two sections and
using the following theorem from the general theory of  continued fractions we will be able
to find a continued fraction expansion (though not a simple one) for Euler's constant $\gamma$
whose $n$th numerator and denominator coincide with the $p_n$ and $q_n,$ respectively.
It seems to be the first non-trivial continued fraction expansion convergent to Euler's constant sub-exponentially,
 the elements  of which can be expressed as a general pattern.

\vspace{0.2cm}

\noindent {\bf Theorem A.}(\cite[Th.~2.2]{jt})
{\it Let $\{A_n\}, \{B_n\}$ be sequences of complex numbers such that
$$
A_{-1}=1, \quad A_0=b_0, \quad B_{-1}=0, \quad B_0=1,
$$
and
$$
A_nB_{n-1}-A_{n-1}B_n\ne 0, \qquad n=0,1,2,\ldots .
$$
Then there exists a uniquely determined continued fraction
$b_0+{\bf K}(a_n/b_n)$ with $n$th numerator $A_n$ and denominator $B_n$ for all $n.$
Moreover,
$$
b_0=A_0, \quad a_1=A_1-A_0B_1, \quad b_1=B_1,
$$
$$
a_n=\frac{A_{n-1}B_n-A_nB_{n-1}}{A_{n-1}B_{n-2}-A_{n-2}B_{n-1}}, \quad
b_n=\frac{A_{n}B_{n-2}-A_{n-2}B_{n}}{A_{n-1}B_{n-2}-A_{n-2}B_{n-1}}, \quad
n=2,3,4,\ldots .
$$
}

\noindent Applying the above theorem to the sequences $\{p_n\}_{n\ge 0}$ and $\{q_n\}_{n\ge 0}$ we get the following.
\begin{theorem}
Euler's constant $\gamma$ has the following continued-fraction expansion:
$$
\gamma=\underset{n=1}{\overset{\infty}{\bf K}}(a_n^{*}/b_n^{*})=
\frac{1}{2}\cf{}{-}\frac{1}{4}\cf{}{-}\frac{5}{16}\cf{}{+}
\frac{36}{59}\cf{}{-}\frac{15740}{404}\cf{}{+}\ldots\cf{}+\frac{a_n^{*}}{b_n^{*}}\cf{}{+}
%\frac{3}{1}\cf{}{+}\frac{4}{1}\cf{}{+}\frac{4}{1}\cf{}{+}
\ldots,
$$
where
\begin{equation}
\begin{array}{lllll}
a_1^{*}=1, \quad & \quad a_2^{*}=-1, \quad & \quad a_3^{*}=-5,
\quad & \quad a_4^{*}=36, \quad & \quad a_5^{*}=-15740, \quad
\\
b_1^{*}=2, \quad & \quad b_2^{*}=4, \quad & \quad b_3^{*}=16,
\quad & \quad b_4^{*}=59, \quad & \quad b_5^{*}=404, \quad
\end{array}
\label{eq116}
\end{equation}
and
\begin{equation}
a_n^{*}=-\frac{(n-1)^2}{4}\Delta_n\Delta_{n-2}, \qquad
b_n^{*}=n^2\Delta_{n-1}+
 \frac{(n-1)(n-2)}{2} q_{n-2},\quad
n\ge 6.
\label{eq116.5}
\end{equation}
 Here $q_n$ is a sequence of positive integers defined by {\rm (\ref{eq99})}
  and $\Delta_n$ is a sequence of integers generated by the
 third-order linear   recurrence:
\begin{equation}
\begin{split}
 (n-1)(n-2)\Delta_{n+2}&=(n-2)(n+1)(n^2+3n-2)\Delta_{n+1} \\
 &-n^2(2n^3+n^2-7n-4)\Delta_n
 +(n-1)^2n^4\Delta_{n-1}, \qquad n\ge 3,
 \end{split}
 \label{eq116.7}
 \end{equation}
 with the initial values $\Delta_1=-1,$ $\Delta_2=-2,$ $\Delta_3=-5,$ $\Delta_4=8.$
Moreover,  $\Delta_n$ is positive for any $n\ge 4,$ and $\Delta_{2n}$ is even for any $n\ge 1.$
\end{theorem}
\begin{proof}
Consider sequences $\{p_n\}_{n\ge 0}$ and $\{q_n\}_{n\ge 0}$ from (\ref{eq99})
and put $p_{-1}=1,$ $q_{-1}=0.$ For $n\ge 0,$ define $\mathfrak{d}_n:=p_{n-1}q_n-p_nq_{n-1}.$
Then we have
\begin{equation}
\mathfrak{d}_0=1, \quad \mathfrak{d}_1=\mathfrak{d}_2=-1, \quad
 \mathfrak{d}_3=-\frac{5}{3}, \quad \mathfrak{d}_4=2, \quad \mathfrak{d}_5=\frac{787}{5}.
 \label{eq113}
\end{equation}
From the recurrent equations (\ref{eq101}), (\ref{eq102}) we get the relation
\begin{equation}
\mathfrak{d}_n=(n-1)^2\mathfrak{d}_{n-1}+\frac{n-2}{n} q_{n-1}, \qquad n\ge 1.
\label{eq114}
\end{equation}
Since $q_n$ is positive for any $n\ge 0$ and $\mathfrak{d}_4>0,$ it follows easily by induction
that $\mathfrak{d}_n$ is positive for any $n\ge 4.$ Taking into account (\ref{eq113}),
we get $\mathfrak{d}_n\ne 0$ for  $n=0,1,2,\ldots .$ Moreover, (\ref{eq114}) implies that
$n\mathfrak{d}_n\in{\mathbb Z}$ for any $n$ and $n\mathfrak{d}_n/2\in{\mathbb Z}$ if $n$ is even.

Now by Theorem A, we get that there exists a uniquely determined continued fraction
$b_0+{\bf K}(a_n/b_n)$ with $n$th numerator $p_n$ and denominator $q_n$ for all $n,$
where $b_0=p_0=0,$ $a_1=p_1-p_0q_1=1,$  $b_1=q_2=1, $ and
\begin{equation}
a_n=-\frac{\mathfrak{d}_n}{\mathfrak{d}_{n-1}}, \qquad
b_n=\frac{p_{n-2}q_n-p_nq_{n-2}}{\mathfrak{d}_{n-1}}, \qquad n\ge 2,
\label{eq115}
\end{equation}
such that
\begin{equation}
\gamma={\bf K}(a_n/b_n)=\frac{1}{2}\cf{}{+}\frac{a_2}{b_2}\cf{}{+}\frac{a_3}{b_3}\cf{}{+}
%\frac{36}{59}\cf{}{-}\frac{15740}{404}\cf{}{+}
\ldots .
\label{eq118}
\end{equation}
From (\ref{eq114}), (\ref{eq115}) we have
\begin{equation*}
a_n=-(n-1)^2-\frac{n-2}{n}\cdot\frac{q_{n-1}}{{\mathfrak d}_{n-1}}, \qquad n\ge 1.
\end{equation*}
From the recurrent equations (\ref{eq101}), (\ref{eq102}) we obtain
$$
p_{n-2}q_n-p_nq_{n-2}=2n {\mathfrak d}_{n-1}+\frac{n-2}{n}\cdot q_{n-2}
$$
and hence
\begin{equation}
b_n=2n+\frac{n-2}{n}\cdot\frac{q_{n-2}}{{\mathfrak d}_{n-1}}, \qquad n\ge 1.
\label{eq117}
\end{equation}
Now let  $\Delta_n:=n {\mathfrak d}_n.$ Then we have $\Delta_n\in {\mathbb Z},$
$\Delta_n$ is positive for  any $n\ge 4$ and $\Delta_{2n}$ is even for any $n\ge 1.$
Define $\rho_0=\rho_1=\rho_2=1,$ $\rho_3=3,$ $\rho_4=10,$
$$
\rho_n=\frac{n(n-1){\mathfrak d}_{n-1}}{2}=\frac{n\Delta_{n-1}}{2}, \qquad n\ge 5,
$$
and make the equivalence transformation of the fraction (\ref{eq118}) by the rule (see
\cite[Theorem 2.6]{jt})
$$
a_n^{*}=\rho_n\rho_{n-1} a_n, \qquad b_n^{*}=\rho_n b_n, \qquad n=1,2,3,\ldots.
$$
Then using first several values of the sequence $q_n,$
$$
q_0=1, \quad q_1=2, \quad q_2=7, \quad q_3=34, \quad q_4=209
$$
and formulas (\ref{eq115}), (\ref{eq117}) we get (\ref{eq116}), (\ref{eq116.5}).
What is left is to show that the sequence $\Delta_n$ satisfies the recurrence (\ref{eq116.7}).
From (\ref{eq114}) we have
$$
q_{n-1}=\frac{\Delta_n-n(n-1)\Delta_{n-1}}{n-2}, \qquad n\ge 3.
$$
Substituting this expression in (\ref{eq101}) we get the four-term recurrence relation (\ref{eq116.7}),
which completes the proof.
\end{proof}

\vspace{0.3cm}

{\bf\small Acknowledgements.} This research was in part supported by grants no.~90110029 (first author) and
no.~90110030 (second author) from the School of Mathematics, Institute for Research in Fundamental Sciences (IPM).
This work was done during our summer visit in 2010 to the Abdus Salam International Centre for Theoretical Physics (ICTP),
Trieste, Italy. The authors wish to thank the staff and, in particular,  the Head of the Mathematics Section of the ICTP,
Professor Ramadas Ramakrishnan,  for their hospitality and the  excellent working conditions. The first author is grateful to
the Commission on Development and Exchanges of the IMU for travel support.

\end{document}